\documentclass{amsart}
\usepackage{amsmath,amsthm,amscd,amsfonts,amssymb, rlepsf}

\usepackage{latexsym}
\usepackage{color} 

\theoremstyle{plain}

\usepackage{epsfig}
\usepackage{graphics}
\usepackage{amsmath}
\usepackage{amscd}
\usepackage{amsfonts,latexsym}
\usepackage{inputenc,fontenc,mathtext}
\theoremstyle{plain}
\newtheorem{theorem}{Theorem}[section]
\newtheorem{lemma}[theorem]{Lemma}

\newtheorem{proposition}[theorem]{Proposition}
\newtheorem{remark}[theorem]{Remark}
\newtheorem{quest}[theorem]{Question}

\theoremstyle{definition}

\theoremstyle{remark}

\newcommand{\Z}{\mathbb{Z}} 
\newcommand{\R}{\mathbb{R}}
\newcommand{\C}{\mathbb{C}}

\newcommand{\co}{\colon\thinspace}

\def\dfn#1{{\em #1}}

 \usepackage{palatino}

\begin{document} 

\title{On generalizing Lutz twists}

\author{John B.\ Etnyre}
\address{
    School of Mathematics,
    Georgia Institute of Technology,
    686 Cherry St.,
    Atlanta, GA  30332-0160}
\email{etnyre@math.gatech.edu}
\urladdr{http://math.gatech.edu/\char126 etnyre}

\author{Dishant M.\ Pancholi}

\address{ Mathematics Group \\
	International Centre for Theoretical Physics\\
	Trieste, Italy }
\email{dishant@ictp.it}



\begin{abstract}

We give a possible generalization  of a Lutz twist to all
dimensions. This reproves the fact that every contact manifold can be given a
non-fillable contact structure and also shows great flexibility in the manifolds
that can be realized as cores of overtwisted families. We moreover show that
$\R^{2n+1}$ has at least three distinct contact structures. \\

\noindent
This version of the paper contains both the texts of the published version of the paper together with an Erratum to the published version appended to the end.

\end{abstract}

\maketitle

\section{Introduction}
Lutz twists have been a fundamental tool in studying contact structures in dimension 3. They have been used to construct contact structures on all closed oriented 3-manifolds and to manipulate the homotopy class of the plane field of a given contact structure. In particular, Lutz \cite{Lutz71} used this construction to give the first proof that all homotopy classes of plane fields can be realized by contact structures. We recall that a Lutz twist alters a contact structure on a solid torus neighborhood of a transverse curve by introducing an $S^1$-family of overtwisted disks (see below for a precise definition). An overtwisted disk in a contact 3-manifold $(M,\xi)$ is an embedded disk in $M$ that is tangent to $\xi$ only along its boundary and at one interior point. A contact manifold is called overtwisted if it contains an overtwisted disk and otherwise it is called tight. Starting with Eliashberg's fundamental paper \cite{Eliashberg92a} defining the tight versus overtwisted dichotomy, these notions have taken a central role in 3 dimensional contact geometry. Overtwisted contact structures are completely classified \cite{Eliashberg89} and exhibit a great deal of flexibility, appearing to be fairly topological in nature. Much of the insight into such structures has come from careful analysis of the Lutz twist construction and natural questions that arose from it.

There is not a great deal known about contact structures in dimensions above 3. Specifically, there are few constructions of contact structures and few tools to manipulate a given contact structure.  We introduce one such tool by giving a possible generalization of a Lutz twist to all odd dimensions. As a consequence we 
reprove and slightly strengthen a result
proved by Niederkr\"uger and van Koert in \cite{NiederkrugerVanKoert07} that every $(2n+1)$--dimensional manifold that has a contact structure can be given a non-fillable 
contact structure. The proof in fact produces an embedded overtwisted family (that is a plastikstufe in the language of \cite{Niederkrueger06}) by changing  the given contact
structure in a small neighborhood of any $(n-1)$--dimensional isotropic
submanifold $B$ (with trivial conformal symplectic normal bundle). The
overtwisted family is modeled on $B$ (that is $B$ is the elliptic singular locus
of the family, see Subsection~\ref{sec:families} below). This construction is
analogous to creating an overtwisted disk  in dimension three by performing a
Lutz twist along a knot in the neighborhood of a point (an overtwisted disk is a
3--dimensional overtwisted family). Overtwisted families, in dimensions above
three, were first considered in \cite{Niederkrueger06} as an obstruction to
symplectic fillability of a contact structure, though precursors of them go back
to Gromov's work \cite{Gromov85}.   In \cite{Presas07} Presas gave the first
examples of overtwisted families in a closed contact manifold of dimension
greater than three.

Our main result is the following. 
\begin{theorem}\label{main}
Let $(M,\xi)$ be a contact manifold of dimension $2n+1$ and let $B$ be an
$(n-1)$--dimensional isotropic submanifold with trivial conformal symplectic
normal bundle. Then we may alter $\xi$ in any neighborhood of $B$ to a contact
structure $\xi'$ that contains an overtwisted family modeled on $B.$ Moreover,
we may assume that $\xi'$ is homotopic to $\xi$ through almost contact
structures.
\end{theorem}
A corollary is the following result originally proven, modulo the statement
about the homotopy class of almost contact structures and core of the
overtwisted family, via a delicate surgery construction in 
\cite{NiederkrugerVanKoert07} based on subtle constructions in \cite{Presas07}.
We call a contact structure \dfn{ps-overtwisted} if it contains an overtwisted family. (We use the prefix ``ps'' in agreement with the literature where it stands for ``plastikstufe''. Though we are using the more descriptive term ``overtwisted family'' instead of ``plastikstufe'', we retain the word ``ps-overtwisted'' for lack of better terminology. Moreover, we can think of the ``ps'' has referring to ``possibly", as it is unclearly if this is the correct generalization of overtwisted to higher dimensions.)
\begin{theorem}\label{maincor}
Every odd dimensional manifold that supports a contact structure also supports
a ps-overtwisted, and hence non (semi-positive) 
symplectically fillable, contact structure in the same homotopy class of almost
contact structure. Moreover, we can assume the overtwisted family is modeled on
any $(n-1)$--dimensional isotropic submanifold with trivial conformal symplectic
normal bundle.
\end{theorem}
We also observe the following non uniqueness result which can also be found in \cite{NiederkrugerPresas08}.
\begin{theorem}\label{nonunique}
There are at least three distinct contact structures on  $\R^{2n+1}, n\geq 1$.
\end{theorem}

We remark that our proof relies on cut-and-paste techniques and branch cover techniques that seem to be new to the literature. These techniques should be useful in constructing contact structures on higher dimensional manifolds and will be more fully explored and systematized in a future paper.

{\em Acknowledgements.} The first author was partially supported by NSF Grant DMS-0804820. We
are grateful for 
valuable comments Klaus Niederkr\"uger made on a first draft of this paper. We
also thank the original referees of the paper who pointed out a gap in the
original proof of our main result and made many other valuable suggestions concerning the manuscript and exposition. We finally thank the second referee who also made many valuable suggestions to improve the paper. 

\section{Background and Notation}
In this section we recall some well known results and establish notation
necessary in the rest of the paper. 
Specifically in Subsection~\ref{sec:nbhds} we prove various Darboux type
theorems about contact structures that agree on compact subsets. In
Subsection~\ref{sec:families} we define overtwisted families. This definition involves
the ``characteristic distribution'' of a submanifold of a contact manifold and has 
various equivalent formulations, just as there are several equivalent definitions of
overtwisted disks in a contact 3-manifold. 
To clarify these equivalent 
formulations we discuss characteristic distributions in Subsection~\ref{sec:distributions}. 
In Subsection~\ref{sec:almost} we
recall a few basic facts about almost contact structures. Finally in
Subsection~\ref{sec:3lutz} we recall the notion of Lutz twist in dimension 3 and
set up notation that will be used in the following sections.

\subsection{Neighborhoods of submanifolds of a contact manifold}\label{sec:nbhds}
A simple application of a Moser type argument yields the following result. 
\begin{proposition}\label{thm:generalnbhd}
Let $N$ be a compact submanifold of $M$ and let $\xi_0$ and $\xi_1$ be two
oriented contact structures on $M$ such that $\xi_0|_N=\xi_1|_N.$ Moreover,
assume we have contact 
forms $\alpha_i$ for $\xi_i$ such that $d\alpha_0|_N=d\alpha_1|_N.$ Then there
are open neighborhoods $U_0$ and $U_1$ of $N$ and a contactomorphism
$\phi:(U_0,\xi_0)\to (U_1,\xi_1)$ that is fixed on $N.$ 
\end{proposition}
\begin{proof}
Set $\alpha_t=t\alpha_1+(1-t)\alpha_0.$ Noting that
$\ker(\alpha_t)=\ker(\alpha_0)$ along $N$ and $d\alpha_t=d\alpha_0$ on $N$ we
see that $\alpha_t$ is contact 
in some neighborhood of  $N.$ Thus $\xi_t=\ker\alpha_t$ now gives a family of
contact structures in a neighborhood of $N$ that agree along $N.$ A standard
application of Moser's argument \cite{McDuffSalamon98} now gives a family of open
neighborhoods $V_t, V_t'$ of $N$ and maps $\phi_t:V'_t\to V_t$ fixed along $N$
such that $\phi_t^*\alpha_t=h_t\alpha_0$ for some positive functions $h_t.$
Setting $U_0=V_1'$ and $U_1=V_1,$ the map $\phi_1$ is the desired
contactomorphism. 
\end{proof}

Recall if $(M,\xi)$ is a contact manifold with  contact structure $\xi$
and $\alpha$ is a  (locally defined) contact form for $\xi$ then for all $x\in M$ the 2--form
$(d\alpha)_x$ is a symplectic form on $\xi_x.$ Since any other contact form
defining $\xi$ differs from $\alpha$ by multiplication by a positive function 
(we always assume a contact form for $\xi$ evaluates positively on a vector
positively transverse to $\xi$), we see there is a well-defined conformal
symplectic structure on $\xi.$ 

A submanifold $L\subset M$ is called \dfn{isotropic} if $T_xL\subset \xi_x$ for
all $x\in L.$ If $M$ has dimension $2n+1$ then the dimension of an isotropic 
$L$ must be less than or equal to $n$ since $T_xL$ is an isotropic subspace of
the symplectic space $(\xi_x, (d\alpha)_x).$  If the dimension of $L$ is $n$
then $L$ is called \dfn{Legendrian}. Given an isotropic $L$ its \dfn{conformal
symplectic normal bundle} is the quotient bundle with fiber
\[
CSN(L)_x=(T_xL)^\perp/T_xL,
\]
where $(T_xL)^\perp$ is the $d\alpha$-orthogonal subspace of $T_xL$ in $\xi_x.$
One may easily check that $CSN(L)_x$ has dimension $2(n-l)$ 
where $l$ is the dimension of $L$ and as bundles 
\[
T_xL\oplus \xi_x/(T_xL)^\perp \oplus CSN(L)_x\oplus \R\cong \xi_x\oplus \R=T_xM,
\]
where the $\R$ factor can be taken to be spanned by any Reeb field for $\xi.$
(All bundle isomorphisms preserve conformal symplectic structures where they are
defined.) 
One may easily check that the bundle $\xi/(TL)^\perp$ is isomorphic to
$T^*L.$ So the only term on the left hand side that is not determined by the
topology of $L$ is $CSN(L)$ which depends on the isotropic embedding of $L$ in
$(M,\xi).$ We now have the following result that easily follows from the above
discussion and Proposition~\ref{thm:generalnbhd} (once one notices that the
conformal symplectic normal bundles can be identified such that the symplectic
structures induced by given contact forms agree). 
\begin{proposition}[Weinstein 1991, \cite{Weinstein91}]\label{prop:csndet}
Let $(M_0,\xi_0)$ and $(M_1,\xi_1)$ be two contact manifolds of the same
dimension and let $L_i$ be an isotropic submanifold of $(M_i,\xi_i), i=0,1.$
If there is a diffeomorphism $\phi:L_0\to L_1$ that is covered by a conformal
symplectic bundle isomorphism $\Phi:CSN(L_0)\to CSN(L_1)$ 
then there are open sets $U_i$ of $L_i$ in $M_i$ and a contactomorphism
$\overline{\phi}:(U_0,\xi_0)\to (U_1,\xi_1)$ that extends $\phi:L_0\to L_1.$\qed
\end{proposition}

\subsection{Characteristic distributions}\label{sec:distributions}

Let $C$ be a $k$--dimensional submanifold of the $(2n+1)$--dimensional contact manifold $(M,\xi).$ The singular distribution 
\[
(C_\xi)_x=T_xC\cap \xi_x\subset T_xC
\]
is called the \dfn{characteristic distribution}. Where the intersection is 
transverse the distribution has dimension $k-1.$ We say $C$ is a \dfn{foliated
submanifold} if the non-singular (that is transverse) part of $C_\xi$ is
integrable. We say $C$ is a \dfn{maximally foliated submanifold} if it is a
foliated submanifold and the dimension of $C$ is $n+1$ (so all the leaves of
$C_\xi$ are locally Legendrian submanifolds of $(M,\xi)$).

The characteristic distribution can be quite complicated as can be its
singularities. Here we clarify a few points that will show up in the definition
of overtwisted families in the next subsection. This allows for more flexibility
in the definition of overtwisted families which, in turn, makes working with
overtwisted families easier. In particular we consider 
codimension 1 and 2 submanifolds of a maximally foliated submanifold $C$ that
are tangent to $\xi.$ 

By way of motivation we recall the 3-dimensional situation. In particular an
overtwisted disk is usually defined to be a disk $D$ with 
characteristic foliation $D_\xi$ having  $\partial D$ as a leaf and a single
elliptic singularity. Alternately one could ask that there is a single elliptic
singularity and $\partial D$ is an isolated singular set. In particular the
exact form of the foliation near $\partial D$ or whether $\partial D$ is a leaf
or a singular set is irrelevant in the sense that given any overtwisted disk of
a particular form near $\partial D$ we can arrange any other suitable form.
Moreover, on the level of foliations there are many types of elliptic
singularities, but again the exact form is irrelevant for the definition of an
overtwisted disk. We will establish similar results for the characteristic
distribution of a maximally foliated submanifold. 

\subsubsection{Neighborhoods of closed leaves.}
Suppose $L$ is a compact subset of the $(n+1)$-dimensional maximally 
foliated submanifold $C$ of the contact manifold $(M,\xi).$ Further suppose $L$
is tangent to $\xi$ and has dimension $n.$ Thus $L$ is a Legendrian submanifold
and hence has a neighborhood contactomorphic to a neighborhood of the zero
section in the jet space $J^1(L)=T^*L\times \R.$ And thus by
Proposition~\ref{prop:csndet} studying the characteristic distribution on $C$ near $L$
can be done by studying embeddings of $L\times (-\epsilon,\epsilon)$ into 
$J^1(L).$  

Now given any closed manifold $B$ suppose that $L=B\times S^1$ and that $L$ has a neighborhood $N\cong L\times
[-1,1]$ (or if $L$ is the boundary of $C$ then $N\cong L\times [0,1]$) in $C$ with $L=L\times\{0\}.$
Here we consider the situation that $\partial (N-L)$ is transverse to the foliation $C_\xi$ and $N-L$
is (non-singularly) foliated by leaves of the form $B\times \R.$ There are two cases we wish to
consider. The first is if $L$ is also a non-singular leaf of $C_\xi$ and the second is when $L$ consists entirely of singular points of $C_\xi$ and in addition that each leaf in $N-L$ is asymptotic to
$B\times \{\theta\}\subset L$ for some $\theta$ and distinct leaves have
distinct asymptotic limits. By working in $J^1(L),$ one can show that either one of these situations is equivalent to the other (that is given one, you can $C^0$ deform $C$ near $L,$ fixing $L,$ so that you obtain the other). Anytime we see an $L$ as in one of
these situations we say the \dfn{leaves of $C_x$ approach $L$ nicely}. This is
analogous to the situation in dimension 3 where the boundary of a Seifert
surface for a Legendrian knot with Thurston-Bennequin invariant 0 can be taken
to be a leaf of the characteristic foliation or a circle of singularities.

\subsubsection{Singular sets.}
We now consider the case of a submanifold $S$ of $C$ of dimension $n-1$ that
consists entirely of singularities of $C_\xi.$ So $S$ is an isotropic 
submanifold of $(M,\xi).$ We also assume that $S$ is an \dfn{isolated singular
set}, that is there are no other singularities of $C_\xi$ in some neighborhood
of $S.$ We call $S$ \dfn{normally symplectic} if the conformal symplectic normal
bundle is trivial and $T_pS\oplus CSN_p(S)=T_pC$ for all $p\in S.$ Thus we may
find a product neighborhood $N= S\times D^2$ of $S$ in $C$ such that $\{p\}\times
D^2$ is tangent to the conformal symplectic normal bundle along $S.$
In this situation $S$ has a neighborhood in $M$ that is contactomorphic to a
neighborhood of the zero section in $T^*S$ in the contact manifold 
\[
(T^*S\times \R\times D^2, \ker(\lambda_{can}+(dz+r^2\, d\theta))
\]
where $z$ is the coordinate on $\R$ and $D^2$ is the unit disk in the plane
with polar coordinate $(r,\theta).$ (If this is not clear see the proof of
Lemma~\ref{lem:nbhdoffamily} below.) Moreover this contactomorphism takes $C$ to
a submanifold of $T^*S\times \R\times D^2$ that is tangent to the zero section
times $D^2$ along the zero section times $\{(0,0)\}.$ We say $S$ is \dfn{nicely
normally symplectic} if $C$ in $T^*S\times\R\times D^2$ can be parameterized by a
map of the following form
\[
f(p, r, \theta)= (\sigma_0(p), g(p, r,\theta), (r,\theta))
\]
where $\sigma_0$ is the zero section of $T^*S$ and $g:(S\times D^2)\to \R$ is
some function such that $g_p(r,\theta)=g(p,r,\theta)$ has graph tangent to the 0
map at $(0,0)$ for all $p\in S.$ Since $S$ is an isolated singular set it is
easy to see that the foliation induced on each $\{p\}\times D^2$ has a
non-degenerate singularity at the origin (since the disk is tangent to the
conformal symplectic normal bundle there) and moreover the type, elliptic or
hyperbolic, of the singularity cannot change for different $p\in S.$ We call $S$
a \dfn{normally elliptic singular set} provided  the singularity on $\{p\}\times
D^2$ is elliptic. Similarly, we call $S$ a \dfn{normally hyperbolic singular
set} provided the singularity on $\{p\} \times D^2$ is hyperbolic. One may now
easily check that if $S$ is an elliptic singular set then we may isotope $C$
near $S$ such that $C$ is still a maximally foliated submanifold, the topology
of the leaves in $C \setminus S$ has not changed and $C=S\times \{0\}\times
D^2$ in $T^*S\times \R\times D^2.$

\subsection{Overtwisted families}\label{sec:families}
Let $(M,\xi)$ be a contact manifold of dimension $2n+1.$ An \dfn{overtwisted
family modeled on $B$}, a closed $(n-1)$--dimensional manifold, 
 (originally called \dfn{plastikstufe} in \cite{Niederkrueger06}) is an
embedding $P=B\times D^2$ in $M$, where $D^2$ is the unit disk in $\R^2,$ such
that 
\begin{enumerate}
\item the characteristic distribution $P_\xi=TP\cap \xi$ is integrable,
\item $B=B\times \{(0,0)\}$ is an isotropic submanifold and the singular set of $P_\xi$,
\item $B$ is a normally elliptic singular set of $P_\xi,$ 
\item $\partial P=B\times \partial D^2$ is a leaf of $P_\xi$, 
\item all other leaves of $P_\xi$ are diffeomorphic to $B\times (0,1),$ and
approach $\partial P$ nicely near one end and 
approach the normally elliptic singularity $B$ at the other end. 
\end{enumerate}
We sometimes call $B$ the {\em core} of the overtwisted family. It is easy to
see from the discussion in the last section that we may assume that $\partial P$
is also an isolated singular set of $P$ with leaves nicely approaching it, since
given this we can slightly perturb $P$ near $\partial P$ such that $\partial P$
is a non-singular leaf of $P_\xi$ as in the definition above. 

A contact manifold $(M,\xi)$ of dimension $2n+1$ is called \dfn{ps-overtwisted} if
it contains an overtwisted family modeled on any $(n-1)$--dimensional manifold.
It is not clear if this is the correct generalization of overtwisted to higher
dimensional manifolds, though it does have some of the properties of 3 dimensional overtwisted contact manifolds. 
Currently the main evidence that this is a good
generalization of  3-dimensional overtwisted contact structures is the following
theorem.

\begin{theorem}[Niederkr\"uger 2006, \cite{Niederkrueger06}]
If $(M,\xi)$ is a ps-overtwisted contact manifold then it cannot be symplectically
filled by a semi-positive symplectic manifold. 
If the dimension of $M$ is less than 7 then it cannot be filled by any
symplectic manifold. 
\end{theorem} 

Recall that a $2n$--dimensional symplectic manifold $(X,\omega)$ is called
semi-positive if every element $A\in \pi_2(X)$ 
with $\omega(A)>0$ and $c_1(A)\geq 3-n$ satisfies $c_1(A)>0.$ Note all 
symplectic 4 and 6 manifolds are semi-positive as are Stein and exact 
symplectic manifolds. It seems likely that the semi-positivity
condition can be removed, but we do not address that issue here.

\subsection{Almost contact structures}\label{sec:almost}
Recall that an (oriented) almost contact structure is a reduction of the
structure group of a 
$(2n+1)$-dimensional manifold $M$ to $U(n)\times 1,$ that is a splitting of the tangent
bundle $TM=\eta\oplus \R$ where $\eta$ is a $U(n)$ bundle and $\R$ is the
trivial bundle. Clearly a co-oriented contact structure induces an almost
contact structure as it splits the tangent bundle into $\xi\oplus \R.$

In dimension 3 any oriented manifold $M$ has an almost contact structure since
the tangent bundle is trivial and the homotopy classes of almost contact
structures are in one to one correspondence with homotopy classes of oriented plane
fields. In higher dimensions the situation is more difficult. It is known, for
example, that in dimensions 5 and 7 a manifold $M$ has an almost contact
structure if and only if its third integral Stiefel-Whitney class vanishes:
$W_3(M)=0.$ See \cite{Gray59, Massey61}. Of course this condition is equivalent to the
second Stiefel-Whitney class $w_2(M)$ having an integral lift. In dimension 5
the homotopy classes of almost contact structures on a simply connected manifold
are in one to one correspondence with integral lifts of $w_2(M).$ The
correspondence is achieved by sending an almost contact structure to its first
Chern class (recall any $U(n)$-bundle has Chern classes).

\subsection{Three dimensional Lutz twists and Giroux torsion}\label{sec:3lutz}
As we wish to generalize Lutz twists from the 3--dimensional setting we digress for a moment to recall this construction. 
Consider the contact structures $\xi_{std}$ and $\xi_{ot}$ on $S^1\times \R^2$ given, respectively,  by 
\[
\xi_{std}=\ker (d\phi+r^2\, d\theta)
\]
and
\[
\xi_{ot}=\ker(\cos r\, d\phi+r\sin r \, d\theta)
\]
where $\phi$ is the coordinate on $S^1$ and $(r,\theta)$ are polar coordinates
on $\R^2.$ Let $T_{std}(a)$ be the torus $S^1\times \{r=a\}$ in $S^1\times \R^2$  with the contact structure $\xi_{std}$ and $T_{ot}(a)$ the same torus
in $S^1 \times \R^2$ together with the contact
structure $\xi_{ot}.$ Furthermore set $S_{std}(a$) to be the solid torus in $S^1\times
\R^2$ bounded by $T_{std}(a)$ with the contact structure $\xi_{std}$ and $S_{ot}(a)$ the
same torus with contact structure $\xi_{ot}.$
Finally set $A_{std}(a,b)=\overline{S_{std}(b)-S_{std}(a)}$ and similarly for $A_{ot}(a,b).$ If we are only concerned with the solid torus or thickened torus and not the contact structure on it we will drop the subscript from the notation. That is for example $S(a)$ is the solid torus $S^1\times D^2_a$ where $D^2_a$ is a disk or radius $a.$ 

Given any $b>0$ one can use the fact that $r\tan r$ takes on all positive values on $( \pi, \frac {3\pi}{2})$ and $(2\pi,\frac{5\pi}2)$ to see there is a unique $b_{\pi}\in (\pi, \frac {3\pi}{2})$ and
$b_{2\pi}\in(2\pi,\frac{5\pi}2)$ such that the 
characteristic foliation on $T_{std}(b)$ is the same as the
characteristic foliation on
$T_{ot}{(b_{\pi})}$ and $T_{ot}(b_{2\pi}).$  Since the characteristic foliation determines a contact structure
in the neighborhood of a surface, one can find some $a$ with $b-a>0$
sufficiently small and an $a_{\pi}\in(\pi,b_\pi)$ and $a_{2\pi}\in (2\pi, b_{2\pi})$
such that there is a contactomorphism $\psi_\pi,$ respectively $\psi_{2\pi},$ from
$A_{std}(a,b)$ to $A_{ot}(a_{\pi},b_{\pi}),$ respectively $A_{ot}(a_{2\pi},b_{2\pi}).$
Moreover, one may explicitly construct $\psi_\pi$ and $\psi_{2\pi}$ in such a way that $\psi_{2\pi}$ preserves the 
$\phi$ and $\theta$ coordinates and $\psi_\pi$ sends 
them to their negatives. 

Now given a transverse curve $K$ in a contact 3--manifold $(M,\xi)$  there is a
neighborhood $N$ of $K$ in $M$ that is contactomorphic to $S_{std}(b)$ in $(S^1\times
\R^2,\xi_{std}).$ A \dfn{half Lutz twist} on $K$ is the process of changing the contact structure 
$\xi$ by removing $S_{std}(a)\subset N$ from $M$ and gluing in
$S_{ot}({b_{\pi}})$ using $\psi_\pi$ to glue $A_{std}({a,b})\subset (M\setminus S_{std}(a))$ to
$A_{ot}({a_{\pi},b_{\pi}})\subset S_{ot}({b_{\pi}}).$  Similarly a \dfn{Lutz twist} (or
sometimes called \dfn{full Lutz twist}) is performed by gluing $S_{ot}({b_{2\pi}})$ in
place of $N$ using $\psi_{2\pi}.$ The subset $S_{ot}({{\pi}})$ of $S_{ot}({b_{2\pi}})$ is called a \dfn{Lutz tube}. 

We now review a similar construction. Consider the manifold $T^2\times [0,1]$
with coordinates $(\theta, \phi,r).$ A $1$--form of the type $k(r)\, d\phi+l(r)\,
d\theta$ will be contact if $k(r)l'(r)-k'(r)l(r)>0.$ Moreover the contact
structure is completely determined, see \cite{Giroux91, Honda00a}, by the slope
of the characteristic foliation $a=-\frac{l(0)}{k(0)}$ on $T^2\times \{0\}$, the
slope of the characteristic foliation $b=-\frac{l(1)}{k(1)}$ on $T^2\times\{1\}$
and the number, $n,$ of times that $-\frac{l(r)}{k(r)}=a$ for $r\in (0,1).$ (Notice
that the contact condition implies the curve $(k(r),l(r))$ is monotonically
winding around the origin in $\R^2$ and thus that $-\frac{l(r)}{k(r)}$ is
monotonically decreasing with $r$, here of course slope $\infty=-\infty$ is
allowed.) We say that the contact structure has Giroux torsion $\frac n2.$ We
denote the corresponding contact structure by $\xi_{n}^{(a,b)}$ and any contact
form for this contact structure of the form  discussed above by $\alpha_{n}^{(a,b)}.$ 

To connect this new notation to our notation above we notice that $A_{std}({a,b})$
above is contactomorphic to $(T^2\times[0,1],\xi_0^{(a^2, b^2)})$ and $A_{ot}({a,b})$ is
contactomorphic to $(T^2\times [0,1], \xi_0^{(a\tan a, b\tan b)})$ if $b-a<\pi$ or, more
generally, $(T^2\times[0,1], \xi_n^{(a\tan a, b\tan b)})$ if $n\pi<b-a<(n+1)\pi.$

Notice that given a transverse curve $K$ as above, we can find in a neighborhood $S_{std}(b)$ of $K$
a thickened torus $T^2\times [0,1]$  with a contact structure
$\xi_0^{(-\frac 1{n+1},-\frac 1 n)}.$ Let $S^{2\pi}(b)$ be the solid torus with
contact structure obtained from the one on $S_{std}(b)$ by replacing $\xi_0^{(-\frac
1{n+1},-\frac 1 n)}$ by $\xi_2^{(-\frac 1{n+1},-\frac 1 n)}.$  It is easy to
show that $S^{2\pi}(b)$ is isotopic to $S_{ot}({b_{2\pi}})$ by an isotopy fixed on the
boundary. So a full Lutz twist can be achieved by replacing $S_{std}(b)$ with
$S^{2\pi}(b).$ (One may similarly describe a half Lutz twist.) Thus we may think
of performing a Lutz twist as adding Giroux torsion along a compressible torus. 

We end this section by recalling that, up to contactomorphism, the tight contact structures on $T^3$ are 
\[
\xi_n=\ker (\alpha_n=\cos (n\phi)\, d\theta_1+\sin (n\phi)\, d\theta_2)
\]
where $(\phi,\theta_1,\theta_2)$ are the coordinates on $T^3$ and $n$ is positive, \cite{Kanda97}.  Notice that $\xi_n$ is obtained from $\xi_{n-1}$ by adding Giroux torsion. 


\section{Generalized Lutz twists}
\label{sec:gLutz_twist}

An \dfn{isotropically parameterized family of transverse curves} in a
$(2n+1)$-dimensional contact manifold $(M,\xi)$ is a smooth map 
\[
\psi:B\times S^1\to M
\]
such that $\psi(\{p\}\times S^1)$ is a curve transverse to $\xi$ for all $p\in
B$ and $\psi(B\times \{\phi\})$ is an 
isotropic submanifold of $(M,\xi)$ for all $\phi\in S^1.$ We say the family is
\dfn{embedded} if $\psi$ is an embedding. Proposition~\ref{thm:generalnbhd} easily
yields the following result.

\begin{lemma}\label{lem:nbhdoffamily}
Let $(M,\xi)$ be a contact manifold  of dimension $2n+1.$ Suppose we have an
embedded isotropically parameterized family of transverse curves $B\times S^1$
in $(M,\xi),$ where the dimension of $B$ is $n-1.$ Moreover assume that the
isotropic submanifold $B\times \{\phi\}$ has trivial conformal symplectic
normal bundle. Then $B\times S^1$ has a neighborhood $N$ in $(M,\xi),$
contactomorphic to a neighborhood of $Z\times S^1\times\{(0,0)\}$ in the contact manifold 
\[
(T^*B\times S^1\times D^2, \ker(\lambda_{can}+(d\phi+r^2\, d\theta)))
\]
where $Z$ is the zero section in $T^*B,$ $\phi$ is the angular coordinate on $S^1,$ $D^2$ is the unit disk in the
plane with polar coordinate $(r,\theta)$ and $\lambda_{can}$ is the canonical
1-form on $T^*B.$
\end{lemma}
\begin{proof}
Choose a diffeomorphism $f$ from $B\times S^1$ in $M$ to $B\times S^1$ in
$T^*B\times S^1\times D^2$ that respects the product structure. We can choose
the normal bundle $\nu$ to $B\times S^1$ in both manifolds to be contained in
the contact hyperplanes. As the conformal symplectic normal bundle to $B$ is
trivial we have $\nu_x=\xi_x/(T_xB)^\perp\oplus \R^2$ and $\xi_x=T_xB\oplus
\nu_x.$ Thus extending our diffeomorphism $f$ to a neighborhood of $B\times S^1$
we can assume that it takes the contact hyperplanes along $B\times S^1$ in $M$
to the contact hyperplanes along $B\times S^1$ in $T^*B\times S^1\times D^2.$ In
addition, we can scale our bundle map along the conformal symplectic normal
direction such that it actually preserves the symplectic structure induced by
the contact forms. Thus our extension of $f$ can be assumed to preserve the
exterior derivative of our contact forms along $B\times S^1.$ Now
Proposition~\ref{thm:generalnbhd} gives the desired contactomorphic neighborhoods. 
\end{proof}

A neighborhood of $B\times S^1$ as given in Lemma~\ref{lem:nbhdoffamily} is contactomorphic to 
\[
N_\epsilon\times S^1\times D^2_b
\]
where $N_\epsilon$ is a neighborhood of the zero section in $T^*B$ and $D^2_b$
is a disk of radius 
$b.$ Using the notation from Subsection~\ref{sec:3lutz}, this is contactomorphic
to $N_\epsilon\times S_{std}(b).$ Denote by $P$ the smooth manifold $N_\epsilon\times
S(b)$ with no particular contact structure on it. 

\begin{lemma}\label{lem:main}
There is a contact structure on $P$ that agrees with the contact structure 
$\ker ( \lambda_{can} +(d \phi + r^2 d \theta))$ near the boundary and agrees with the one on 
$N_{\epsilon} \times S_{ot}(b_{2\pi})$ on $N_{\epsilon ''}\times S(b)$ for some positive
$\epsilon ''\ll\epsilon$. 
\end{lemma}
We define the \dfn{(generalized) Lutz twist} of $(M,\xi)$ along $B \times S^1$
to be the result 
of removing $N_\epsilon \times S(b)$, with the standard contact structure, from $M$ and replacing it with the contact
structure constructed in the lemma.

We call the contact manifold $P,$ with the contact structure described in Lemma~\ref{lem:main}, a
\dfn{Lutz tube with core $B$}. Given that there is a  Lutz tube with core
$B$  as claimed in Lemma~\ref{lem:main} the following theorem is almost
immediate.


\begin{theorem}\label{thm:genlutz}
Let $(M,\xi)$ be a contact manifold  of dimension $2n+1.$ Suppose we have an
embedded isotropically parameterized family of transverse curves $B\times S^1$
in $(M,\xi),$ where the dimension of $B$ is $n-1.$ 
 Moreover assume that the
isotropic submanifold $B\times \{\phi\}$ has trivial conformal symplectic
normal bundle.
Then we may alter $\xi$ in
any neighborhood of $B\times S^1$ to a contact structure $\xi'$ that is
ps-overtwisted. Moreover there is an $S^1$-family of overtwisted families modeled
on $B$ in $\xi'.$
\end{theorem}

To see how this theorem generalizes a Lutz twist in dimension 3, notice that the
only possibility for a connected $B$ in a contact 3-manifold $(M,\xi)$ is
$B=\{pt\}.$ So the embedded isotropically parameterized family of
transverse curves in this case is simply a transverse knot $K\subset M.$ Clearly
the Lutz tube $S_{ot}({\pi})$ is an $S^1$-family of overtwisted disks. 

\begin{proof}
The modification mentioned in the theorem is, of course, a (full) Lutz twist. It
is clear that this can be performed in any arbitrarily small neighborhood of
$B\times S^1.$ We are left to check that we have an $S^1$-family of
the embedded overtwisted families modeled on $B,$ but this is obvious as one
easily checks that $B\times \{\phi\}\times D^2_{\pi},$ where $B$ is thought of
as the zero section of $T^*B,$ is an overtwisted family modeled on $B$ for each
$\phi\in S^1$ contained in $P$.
\end{proof}

\subsection{Preliminaries for constructing a Lutz tube with the core $B$}\label{sec:setup}

The purpose of this section is to motivate  as well as set up the preliminaries for the proof of
 Lemma~\ref{lem:main} which will establish the existence of the Lutz tube with
core $B.$ We start  by setting up some preliminary notation (and will use notation established in Subsection~\ref{sec:3lutz}). 
 

Suppose we are given the standard contact structure
$\xi_{std}=\ker\alpha_{std},$ where $\alpha_{std}= d\phi +r^2\, d\theta$ on
$S(\delta)=S^1\times D^2_\delta.$ We can choose $0<\delta''< \delta'<\delta$ and set 
\[
\alpha_{ot}= \ker (k(r)\, d\phi + l(r)\, d\theta)
\]
where $k$ and $l$ are chosen such that $l'(r)k(r)-k'(r)l(r)>0,$ $k(r)=1$ and 
$l(r)=r^2$ for $r\in[0,\delta'']\cup [\delta',\delta]$ and the curve
$(k(r),l(r))$ winds around the origin once as $r$ runs from $0$ to $\delta.$ 


Let $N_\epsilon$ be a neighborhood of the zero section $Z\subset T^*B$ and
denote the Liouville form by $\lambda_{can}.$ Choose some
$0<\epsilon''\ll\epsilon'< \epsilon.$ Recall we want to replace
$\ker(\lambda_{can}+\alpha_{std})$ on $N_{\epsilon''}\times S(\delta)$ with $\ker(\lambda_{can}+\alpha_{ot}).$  

We begin to define a contact structure $\xi$ on $N_{\epsilon}\times S(\delta)$ as follows. 
\[
\xi=
\begin{cases}
\ker(\lambda_{can}+\alpha_{std})& \text { on } \left(\overline{(N_{\epsilon}-N_{\epsilon'})}\times S(\delta)\right) 
				\cup \left( N_{\epsilon}\times A(\delta', \delta)\right)
			  \cup \left( N_\epsilon \times S(\delta'')\right)\\
\ker(\lambda_{can}+\alpha_{ot})& \text{ on } N_{\epsilon''}\times S({\delta}).
\end{cases}
\]

\begin{figure}[ht]
  \relabelbox \small {
  \centerline{\epsfbox{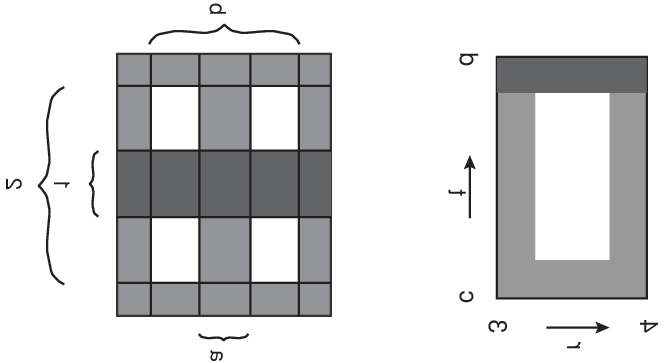}}}
  \relabel{1}{$N_{\epsilon''}$}
  \relabel{2}{$N_{\epsilon'}$}
  \relabel{a}{$S({\delta''})$}
  \relabel{b}{$S({\delta'})$}
    \relabel{3}{$\delta''$}
  \relabel{4}{$\delta'$}
  \relabel{t}{$t$}
  \relabel{r}{$r$}
  \relabel{c}{$a$}
  \relabel{d}{$b$}
    \endrelabelbox
        \caption{On the left is the manifold $N_\epsilon\times S(\delta).$ The lighter shaded regions are where the contact form is given by $\ker(\lambda_{can}+\alpha_{std})$ and the darker shaded regions are where it is given by $\ker(\lambda_{can}+\alpha_{ot}).$ The contact structure needs to be extended over the unshaded regions. On the right is the $rt$-coordinates of the manifold $\overline{(N_{\epsilon'}-N_{\epsilon''})}\times A(\delta'',\delta')$ written as $W\times[a,b]\times
[\delta'',\delta']\times S^1\times S^1.$ The lighter shaded region is where the contact form is given by $\lambda +e^t\alpha_{std}$ and on the darker shaded region the contact form is given by  $\lambda +e^t\alpha_{ot}.$ The contact structure needs to be extended over the unshaded region.}
     \label{fig-1}
\end{figure}

The left hand side of Figure~\ref{fig-1} shows the region where the contact structure is already
defined. Notice that $\ker(\lambda_{can}+\alpha_{std})$ and
$\ker(\lambda_{can}+\alpha_{ot})$ agree on $N_{\epsilon''}\times A(\delta',\delta)$  and on $N_{\epsilon''}\times S(\delta'')$ and thus $\xi$ is well defined where it
is defined. 
Notice that in the case when $B$ is a point we have already defined the contact structure on all of $P$ and it clearly
corresponds to the (full) Lutz twist. 
When $B$ has positive dimension we claim that $\xi$ may be extended over the rest of
$N_{\epsilon}\times S(\delta).$  That is we need to extend $\xi$ over
$\overline{(N_{\epsilon'}-N_{\epsilon''})}\times A(\delta'',\delta').$ 

To this end we notice that if we denote by $\lambda$ the 1--form $\lambda_{can}$
restricted to the unit cotangent bundle $W$ of $T^*B$ then there is a diffeomorphism from  $\overline{(N_{\epsilon'}-N_{\epsilon''})}$ to $[\epsilon'',\epsilon']\times W$ that takes the 1--form $\lambda_{can}$ to $t\lambda,$ where $t$ is the coordinate on $[\epsilon'',\epsilon'].$ Moreover, $A(\delta'',\delta')= S^1\times
\overline{(D^2_{\delta'}-D^2_{\delta''})}$ can be written $S^1\times
[\delta'',\delta']\times S^1,$ with coordinates $(\phi, r,\theta).$ Setting $a= -\ln \epsilon'$ and $b=-\ln \epsilon''$ we can
write  $\overline{(N_{\epsilon'}-N_{\epsilon''})}\times A(\delta'',\delta')$ as  $W\times[a,b]\times
[\delta'',\delta']\times S^1\times S^1.$ This identification is orientation
preserving where we have sent $t$ to $e^{-t}$ and the last three coordinates are
$(r,\theta, \phi).$ 
Near $t=a, t=b$ and $r=\delta'', \delta'$ we have (the
germ of) a contact form defined as the kernel of $e^{-t}\lambda+\alpha$ where
$\alpha$ is either $\alpha_{std}$ or $\alpha_{ot}.$ Of course this contact
structure is also defined by $\lambda+e^{t}\alpha.$ So we see that we need to
construct a contact structure on 
$\overline{(N_{\epsilon'}-N_{\epsilon''})}\times A(\delta'',\delta')$ that is equal to the kernel of
$\lambda$ plus the symplectization of $\alpha_{std}$ near $t=a$ and $r=\delta'',
\delta'$ and equal to the kernel of $\lambda$ plus the symplectization of
$\alpha_{ot}$ near $t=b.$  See the right hand side of Figure~\ref{fig-1}.

More specifically we can assume that the neighborhoods where $\alpha_{ot}$ and
$\alpha_{std}$ agree and the coordinates on $S^1\times S^1$ are chosen in such a
way that near $\{a\}\times [\delta'',\delta']\times S^1\times S^1$ the 1--form
is diffeomorphic to $e^t\alpha_0^{(0,\infty)}.$ (See Section~\ref{sec:3lutz} for
the notation being used here. Also notice that we have chosen coordinates on $S^1\times S^1$ so that $\alpha_0^{(0,\infty)}$ is the appropriate form to use, instead of the form $\alpha_0^{((\delta'')^2,(\delta')^2)}$ which we would have to use if not for the coordinate change. This simplifies notation and makes the construction easier to follow. To see that such a choice of coordinates is possible we notice that $\delta''$ and $\delta'$ can be chosen so that the the smallest integral vectors spanning the characteristic foliations on $T_{std}(\delta'')$ and $T_{std}(\delta')$ from an integral basis for $\Z^2.$)  Similarly near $\{b\}\times 
[\delta'',\delta']\times S^1\times S^1$ the 1--form is diffeomorphic to
$e^t\alpha_2^{(0,\infty)}.$

Now to  motivate a possible approach for extending the contact structure $\xi$
over all of  $N_{\epsilon}\times S(\delta),$ we notice that {\em if}
we could construct an exact symplectic structure $d\beta$ on $[0,1]\times
[0,1]\times T^2$ such that near  $\{0\}\times  [0,1]\times T^2$ and $[0,1]\times \{0,1\}\times T^2$ 
the $1$--form
$\beta=e^t\alpha_0^{(0,\infty)}$ and near $\{1\}\times[0,1]\times T^2$ we have
$\beta=e^t\alpha_2^{(0,\infty)}$ 
then we could extend $d\beta$ to an exact symplectic structure on
$\R\times [0,1]\times T^2$ that looks like the symplectization of
$\alpha_0^{(0,\infty)}$ for negative $t$ and like the symplectization of
$\alpha_2^{(0,\infty)}$ for $t$ larger than $1$. By  rescaling the exact
symplectic form if necessary and choosing $a$ and $b$ sufficiently far apart
(notice we can clearly do this as $N_\epsilon$ minus the zero section is exact
symplectomorphic  to $(-\infty, c)\times W$ for some $c$) there will be some
subset of $W\times \R\times [0,1]\times T^2$ with contact form $\lambda+\beta$
that has a neighborhood of its boundary contactomorphic to a neighborhood of the
boundary of $W\times [a,b]\times [\delta'',\delta']\times S^1\times S^1.$  Of
course it is well known that such an exact symplectic cobordism (even non-exact)
cannot exist as it would allow one to construct symplectic fillings of
overtwisted contact structures. But we will show below that there is a contact
structure on $W\times [0,1]\times[0,1]\times T^2$ that looks like this one near
the boundary and hence we can finish the argument as above.

\subsection{A construction of a Lutz tube with the core $B$}

From the discussion at the end of the last section it is immediate that Lemma~\ref{lem:main}, and hence the existence of a Lutz tube with core $B,$
will be established once we demonstrate the following lemma. 

\begin{lemma}\label{RMN} 
In  the notation of Section~\ref{sec:3lutz}, there is a contact structure on $W \times ([0,1] \times [0,1])  \times T^2$ such that the following properties are satisfied:
\begin{enumerate}
\item near $W\times \{0\}\times [0,1]\times T^2$  and $W\times [0,1]\times  \{0,1\}\times T^2$ the contact structure is  contactomorphic to $\lambda+e^t \alpha_0^{(0,\infty)},$ and
\item near $W\times \{1\}\times [0,1]\times T^2$ the contact structure is  contactomorphic to $\lambda+e^t \alpha_2^{(0,\infty)}.$
\end{enumerate}
\end{lemma}

We establish Lemma~\ref{RMN} by first considering a similar lemma for the  
manifold $W \times [0,1] \times T^3$. Here the new manifold can be thought of as
being  obtained from the old one by identifying the boundary of $[0,1] \times
T^2$ by the identity. Once we prove the lemma stated below, Lemma~\ref{RMN} will
follow by removing a suitable portion of the manifold 
$W \times [0,1] \times T^3.$

\begin{lemma}\label{lem:3.5}
There exists a contact structure on $W\times ([0,1]\times S^1)\times T^2$ such 
that near one boundary component $W\times \{0\}\times S^1\times T^2$ the contact
structure is contactomorphic to $\lambda + e^t \alpha_1$ and near the boundary
component $W\times \{1\}\times S^1\times T^2$ it is contactomorphic to $\lambda+
e^t\alpha_2.$ Here $\alpha_1$ and $\alpha_2$ are the contact forms on $T^3$
defined in  Section~\ref{sec:3lutz}. 
\end{lemma}


\begin{proof}
Consider the cotangent bundle $T^*T^2=\R^2\times T^2$ with coordinates $(p_1,p_2,\theta_1,\theta_2).$ The
1--form $\beta=p_1\, d\theta_1+ p_2\, d\theta_2$ is the 
primitive of the symplectic form $d\beta$ on $\R^2\times T^2.$ Moreover given any point $p$ in $\R^2$
the lift of the radial vector field $v_p$ centered at $p$ in $\R^2$ to $\R^2\times T^2$ is an
expanding vector field for $d\beta.$ Let $X$ be the disk of radius, say, 10 in
$\R^2\times T^2,$ and let $p=(5,0)\in \R^2.$ The expanding vector field $v_p$
is transverse to $\partial X$ so $\partial X=T^3$ is a hypersurface of contact
type.  Note $T^3$ in this context is naturally thought of as $S^1\times T^2$
with coordinates $(\phi,\theta_1,\theta_2).$ The contact structure induced on
$\partial X$ is easily seen to be $f(\phi)\, d\theta_1+ g(\phi)\, d\theta_2$
where $(f(\phi),g(\phi))$ parameterize an ellipse about the origin in $\R^2.$
Thus the contact structure on $T^3$ is the unique strongly fillable contact
structure $\xi_1$ on $T^3$ (of course this is also obvious since $X$ is a
strong filling of the contact structure). Similarly if $X'$ is a disk of radius
one about $p$ times $T^2$ then it is also a strong symplectic filling of
$(T^3,\xi_1).$ Moreover, $\overline{X-X'}$ is an exact symplectic cobordism
from the symplectization of $(T^3,\xi_1)$ (this is the boundary component of
$\overline{X-X'}$ coming from the boundary of $X'$) to the
symplectization of $(T^3,\xi_1)$ (this is the boundary component of $\overline{X-X'}$ that
is also the boundary of $X$).  See the left hand side of Figure~\ref{fig-2}.

\begin{figure}[ht]
  \relabelbox \small {
  \centerline{\epsfbox{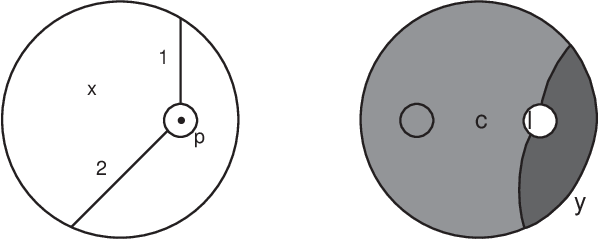}}}
  \relabel{1}{$\gamma_1$}
  \relabel{2}{$\gamma_2$}
  \relabel{x}{$X$}
  \relabel{p}{$X'$}
  \relabel{c}{$C$}
  \relabel{y}{$Y''$}
  \relabel{l}{$X_1'$}
    \endrelabelbox
        \caption{The disks $X$ and $X''$ in the $p_1p_2$-plane are shown on the left, together with the point $p$ and the curves $\gamma_1$ and $\gamma_2$ used in the proof of Lemma~\ref{RMN}. On the right the $p_1p_2$-part of $Y'$ and $X_1'$ is shown. The entire shaded region is the $p_1p_2$-part of $Y''.$ The lighter shaded region is the $p_1p_2$-part of the manifold $C$ constructed in the proof of Lemma~\ref{RMN}. }
     \label{fig-2}
\end{figure}

Let $Y=W\times X$ and $\eta=\lambda+\beta.$ Clearly $\eta$ is a contact form on $Y.$ We now need an auxiliary observation whose proof we give below.
\begin{lemma}\label{lem:aux}
In any open neighborhood of $S=W\times\{(0,0)\}\times T^2$ in $Y,$ $S$ may be
isotoped such  that it is a contact submanifold of $Y.$ 
\end{lemma}


It is well known, and easy to prove \cite{Geiges08b}, that if $\pi:Y'\to Y$ is any
two-fold cover branched over a contact submanifold $S$ then $Y'$ has a
contact structure that agrees with the pullback contact structure away from a
neighborhood of the branch locus. We briefly recall this general construction as the
details will be needed below. If $\alpha$ is a contact form on $Y$ then
$\pi^*\alpha$ is a contact form on $Y'$ away from the branch locus. Near the
branch locus let $\gamma$ be a connection $1$--form on the normal circle bundle
and let $f:Y'\to \R$ be a function that is 1 near the branch locus, zero outside a
slightly larger neighborhood and radially decreasing in between. For small $\epsilon$ the
form $\pi^*\alpha +\epsilon f(x) (d(x))^2\gamma$  where $d(x)$ is the distance
from $x$ to the branch locus, can easily be shown to be contact. 

Set
$Y''=\overline{Y'-X'_1}$  where $X'_1$ is one of the two connected components
of $\pi^{-1}(W\times X').$ See the right hand side of Figure~\ref{fig-2}.
Notice that $\partial Y''$ has two components $B_1\cup B_2.$ A neighborhood of
one of them, $B_1$ say, is also the boundary of $W\times X'_1$ and it is
clearly contactomorphic to a neighborhood of $W\times \{0\}\times T^3$ in $W\times[0,1]\times T^3$ with the contact form
$\lambda+e^t \alpha_1$ and $B_1=W\times\{0\}\times T^3.$ The other boundary
component $B_2$ is also the boundary of $Y'$ and has a neighborhood
contactomorphic to a neighborhood of $W\times \{1\}\times T^3$ in $W\times [0,1]\times T^3$ with contact form
$\lambda+e^t\alpha_2$ and $B_2=W\times\{1\} \times T^3.$ Thus $Y''$ is the
desired contact manifold. 
\end{proof}

We now establish our auxiliary lemma.
\begin{proof}[Proof of Lemma~\ref{lem:aux}]
We adapt a construction from \cite{Bourgeois02}. In particular see that paper
for more on open book decompositions, 
but briefly an open book decomposition of a manifold $M$ is a codimension two
submanifold $N$ with trivial normal bundle together with a locally trivial fibration
$\pi:(M-N)\to S^1$ such that the closure of the fibers are submanifolds of $M$
whose boundaries are $N.$ A contact structure $\xi$ is supported by the open
book decomposition if there is a contact 1--form $\alpha$ such that $\alpha$ is
contact when restricted to $N$ and $d\alpha$ is a symplectic form on the fibers
of $\pi.$ In addition, we also need  the orientation induced on $N$ by $\alpha$
and on $N$ by the fibers, which are in turn  oriented by $d\alpha,$ to agree. As
every contact structure $\xi$ on $M$ has an open book decomposition supporting
it, see \cite{Giroux02}, it
was shown in \cite{Bourgeois02} that one can use this structure to construct a
contact structure on $M\times T^2.$ We show that the construction in that paper
can be used to give the desired embedding of $W\times T^2$ into $Y.$

Let $(N,\pi)$ be an open book decomposition of $W$ that supports the contact
structure $\ker \lambda.$ So $\pi:(W-N)\to S^1$ is a fiber bundle. Choose a function $f:[0,\infty)\to \R$ that is equal to the identity near $0,$
increasing near $0$ and then constant. 
Fixing a
metric on $W,$ let $\rho$ be the distance function from $N$ and set
 $F:W\to \R^2$ to be
$F(x)=f(\rho(x))\pi(x).$ Denote the coordinate functions of $F$ by $F_1$ and
$F_2.$ (Even though $\pi(x)$ is not defined for $x\in N$ we can still define
$F(x)$ since $f(\rho(x))=0$ there.) We define the embedding of $W\times T^2$
into $W\times X'$ by,
\[
\Phi(x, \theta_1,\theta_2)= (x, F_1(x),F_2(x), \theta_1,\theta_2).
\]
Notice that $\Phi^*(\lambda+\beta)=\lambda+F_1(x)\, d\theta_1+ F_2(x)\,
d\theta_2.$ 
One may easily check that this is a contact form, or see \cite{Bourgeois02}. 
\end{proof}

We finish the proof of Lemma~\ref{lem:main} by establishing  Lemma~\ref{RMN}.

\begin{proof}[Proof of Lemma~\ref{RMN}]
Set $\gamma_1= \{(p_1,p_2) :p_1=5, p_2\geq0 \}\cap X$ and $\gamma_2= \{(p_1,p_2) :p_1=p_2+5, p_2\leq 0 \}\cap X.$ Notice that $\gamma_i, i=1,2,$ is a flow line of $v_p.$ 
We consider the lift of $H_i=W\times \gamma_i\times T^2$ to $Y'$ that 
intersects $X'_1.$ Let $C$ be the closure of
the component of $Y''-(H_1\cup H_2)$ containing the branch locus. Notice that topologically $C$ is $W\times
[0,1]\times [0,1]\times T^2.$ One easily sees that a neighborhood of $W\times
\{0\}\times [0,1]\times T^2$ is contactomorphic to the kernel of $\lambda$ plus
the symplectization of the 1--form $\alpha_0^{(0,1)}$  on $[0,1]\times
T^2.$ Similarly a neighborhood of $W\times \{1\}\times [0,1]\times T^2$ is
contactomorphic to the kernel of $\lambda$ plus the symplectization of the
1--form $\alpha_2^{(0,1)}.$ Moreover, since the branched covering map is a 
local diffeomorphism away from the branch locus we see that  near $W\times [0,1]\times
\{0,1\}\times T^2$ the contact form is $\lambda$ plus the symplectization of the 1--form $\alpha_0^{(0,1)}.$
Notice that there is a diffeomorphism of $[0,1]\times T^2$ that takes $\alpha_k^{(0,1)}$ to $\alpha_k^{(0,\infty)}.$ We can use this diffeomorphism  to pull back the contact structure just constructed to the contact structure described in the lemma. 
\end{proof}


\section{From isotropic submanifolds to parameterized families of transverse
curves}\label{sec:getfamily}

In this section we prove our main theorem concerning Lutz twists by finding
an embedded isotropically parameterized 
family of transverse curves given an isotropic submanifold of dimension $n-1$ in
a $(2n+1)$--dimensional  contact manifold. 

\begin{lemma}\label{lem:nbhd}
Let $(M,\xi)$ be a contact manifold  of dimension $2n+1.$
If $B$ is an $(n-1)$--dimensional isotropic submanifold of a contact manifold
$(M,\xi)$ with trivial conformal symplectic normal bundle then $B$ has a
neighborhood contactomorphic to a neighborhood of the zero section of $T^*B$ in
the contact manifold  
\[
(T^*B\times \R^3, \ker(\lambda_{can}+(dz+r^2\, d\theta))),
\]
where $(r,\theta, z)$ are cylindrical coordinates on $\R^3$ and $\lambda_{can}$ is the canonical 1--form on $T^*B.$
\end{lemma}

\begin{proof}
By Proposition~\ref{prop:csndet} we know that the conformal symplectic normal
bundle of an isotropic submanifold determines the 
contact structure in a neighborhood of the submanifold. The lemma follows as the
zero section of $T^*B$ sitting inside $T^*B\times \R^3$ clearly has trivial
conformal symplectic normal bundle (in fact it can be readily identified with
$\ker(dz+r^2\, d\theta)$ in $T\R^3\subset T(T^*B\times \R^3)$ which along the zero section is just the $r\theta$--plane in $\R^3$).
\end{proof}

Suppose $B$ is an $(n-1)$-dimensional isotropic submanifold of a contact
manifold $(M,\xi)$ with 
trivial conformal symplectic normal bundle. Let $N$ be the neighborhood of $B$
in $M$ that is contactomorphic to a neighborhood of the zero section in $T^*B\times
\R^3$ given in Lemma~\ref{lem:nbhd}. If $K$ is any transverse curve in
$(D^3,\ker(dz+r^2\, d\theta))$ then consider $B'=B\times K$ in $T^*B\times
D^3\cong N\subset M$ where we are thinking of $B$ as the zero section of $T^*B.$
Clearly $B'$ is an embedded isotropically parameterized family of transverse
curves. 


We are now ready to prove our main result. 

\begin{proof}[Proof of Theorem~\ref{main}]
Given an $(n-1)$--dimensional isotropic submanifold $B$ of $(M,\xi)$ with
trivial conformal symplectic normal bundle, Lemma~\ref{lem:nbhd} and the
discussion above yield the desired embedded isotropically parameterized family
of transverse curves $B'=S^1\times B$ in any arbitrarily small neighborhood of
$B.$ Theorem~\ref{thm:genlutz} now allows us to perform a generalized Lutz twist
on $B'.$

In order to show that the homotopy class of almost contact structure is unchanged during this operation
we will construct a 1--parameter family of almost contact structures that starts with the contact structure obtained by performing a generalized Lutz twist and ends with the original contact structure. 
The homotopy will be through confoliations. In the construction we
use  the notation from  Subsection~\ref{sec:setup} and the  proof of
Lemma~\ref{lem:aux}. 


We begin by making a preliminary isotopy of our contact structures that will simplify our argument later. 
Denote by $\beta_{ot}$ the contact form  on $P = N_{\epsilon} \times S(b)$ satisfying the properties described in Lemma~\ref{lem:main} and 
let $\beta_{st}$ denote the standard contact form on $P$. Let $k(t,r)$ be a function with support in $[0,\epsilon'')\times (\delta'', \delta')$ where it is strictly positive. 
Consider the forms $\beta_{st} + k(t,r)\, dr$ and
$\beta_{ot} + k(t,r)\, dr$. Notice that if the supremum of $k$ is sufficiently small they define contact structures isotopic to the contact structures defined by the contact forms $\beta_{st}$ and $\beta_{ot}$, respectively. 
Abusing notation slightly we will continue to denote these forms by $\beta_{st}$ and $\beta_{ot}$, respectively.  


We will break the rest of the argument into three steps. In the first step we homotope the Lutz twisted contact structure (all homotopies will be through almost contact structures) to a confoliation on 
$W\times[a,b]\times[\delta'',\delta']\times T^2$ given as the kernel of a particularly simple 1--form. This homotopy is fixed near the boundary so extends over all of $M.$ 
In the second step we further homotope the confoliation to agree with $\xi$ on $W\times[a,b]\times[\delta'',\delta']\times T^2$ in such a way that the homotopy clearly extends over all of 
$M$ except $N_{\epsilon''}\times S(\delta).$  In the last step we complete the homotopy to $\xi$ by extending the homotopy from step two to $N_{\epsilon''}\times S(\delta).$ 

\smallskip
\noindent
{\bf Step 1:} 
In this step we
homotope our contact structure to a confoliation on
$W\times[a,b]\times[\delta'',\delta']\times T^2$ that is given as the kernel of
a particularly simple 1--form. 
To this end we set $S_s$ to be the embedding of $W\times T^2$
into
$Y$ given by 
\[
\Phi_s(x, \theta_1,\theta_2)= (x, sF_1(x),sF_2(x), \theta_1,\theta_2),
\]
where we use the notation from Lemma~\ref{lem:aux}. Using a Riemannian metric on
$Y$ we can extend this to a $1$--parameter family of embeddings of the normal
disk-bundle $W\times D^2\times T^2$ into $Y.$ (Notice that we can assume each of
the disk-bundles have the same radius $r<1.$) Let $\pi_s:Y'_s\to Y$ be the two
fold branched cover of $Y$ over $S_s.$ We have the 1--parameter family of
1-forms $\alpha_s=\pi_s^*\alpha$ on $Y'.$ Fix a function $f:[0,1]\to [0,1]$
that is $1$ near $0$, $0$ past $r$ and decreasing elsewhere. Let
$\zeta_s=f(d_s(x)) (d_s(x))^2 \gamma_s$ where $d_s:Y'\to \R$ is the distance
from the branch locus of $\pi_s,$ and $\gamma_s$ is the connection $1$--form on
the normal disk bundle to the branch locus. We can extend $\zeta_s$ to all of
$Y'$ and they will be a smooth family
of $1$--forms.  As discussed above, or see \cite{Geiges08b}, for small enough
$c_s$ the 1--form $\alpha_s+c_s\zeta_s$ will be contact for $s\not=0.$ We can
choose the $c_s$ smoothly so that they are decreasing with $s$ and $c_0=0.$ Thus
$\xi_s=\ker (\alpha_s+c_s\zeta_s)$ is a 1--parameter family of hyperplane fields
on $Y'$ that are contact for $s\not=0$ and give a confoliation for $s=0$. We
claim that $\xi_0$ has an almost complex structure that makes it into an almost
contact structure that is homotopic through almost contact structures to the
almost contact structures on $\xi_s$ for $s\not=0.$ To see this fix a metric on
$Y'$ and let $v$ be the (oriented) unit normal vector to $\xi_0.$ Since $\xi_s$
is a smooth family of hyperplane fields there is some small $s$ such that
$\xi_s$ is also transverse to $v.$ We can now project $\xi_s$ along $v$ onto
$\xi_0.$ This projection will be a bundle isomorphism $\xi_s\to \xi_0$ thus we
can use it to define an almost complex structure on $\xi_0.$ Similarly if we
take $\xi'_u=\ker ((1-u)(\alpha_s+c_s\zeta_s)+u(\alpha_0+c_0\zeta_0)),$ for $u\in [0,1],$ then we can use this
projection to define an almost complex structure on $\xi'_u$ for all $u.$ That
is $\xi_0$ with this almost complex structure is homotopic through almost
contact structures to $\xi_s.$ We notice that 
$\alpha_0$ is $\lambda+\beta$ where $\beta$ is a 1--form on $D^2\times T^2.$ 
(Recall that $Y'=W\times D^2\times T^2.$) So $\xi_0$ can be decomposed as $\xi'\oplus
D$ where $D$ is a 4--dimensional distribution. The projection map $p:Y'\to
D^2\times T^2$ maps $D$ isomorphically onto the tangent space of $D^2\times
T^2.$ We can use this isomorphism to put an almost complex structure on
$D^2\times T^2$ which will be used below. 

We have homotoped our contact structure to an almost contact structure on 
$W\times[a,b]\times[\delta'',\delta']\times T^2$ that is given as the kernel of
$\lambda+\beta,$ but notice that the homotopy is fixed near the boundary of our
manifold so this is a homotopy of almost contact structures on our entire
manifold.  

\smallskip
\noindent
{\bf Step 2:}
We will further homotope the almost contact structure
on $W\times[a,b]\times[\delta'',\delta']\times T^2$ to be
$\lambda+e^t\alpha_{std}.$ For this consider the 
$1$--parameter family of $1$--forms $\beta_s=\lambda+ e^t \left(
s\alpha_{std}+(1-s)\beta\right).$ 
Notice that $\xi_s=\ker\beta_s$ is always a hyperplane field and
$\xi_s=\xi'\oplus D_s$ where $\xi'=\ker\lambda$ is  the contact structure on
$W$ and $D_s$ is a 4--dimensional distribution. Moreover if
$p:W\times[a,b]\times[\delta'',\delta']\times T^2\to
[a,b]\times[\delta'',\delta']\times T^2$ is the projection, then $dp$ is an
isomorphism from $D_s$ to the tangent space of
$[a,b]\times[\delta'',\delta']\times T^2.$ We can use this isomorphism to induce
an almost complex structure on $D_s$ for all $s$ and thus $\xi_s$ is an almost
contact structure (as $\xi'$ clearly has an almost complex structure since it is
contact). 

\smallskip
\noindent
{\bf Step 3:}  We are left to extend the homotopy above over the region 
$N_{\epsilon''}\times S(\delta).$ 
Notice that we can assume that there is some
$\eta$ such that 
\[
W\times[a,a+ \eta]\times[\delta'',\delta']\times T^2= \left(W\times[a,b]\times[\delta'',\delta']\times T^2\right)\cap \left(N_{\epsilon''}\times  S(\delta)\right).
\]
Thus we already have our homotopy defined on part of $N_{\epsilon''}\times
D_\delta.$ 
If we consider the $1$--forms $\beta_s=\lambda_{can}+(s\alpha_{std}+
(1-s)\alpha_{ot}) + k(t,r)\, dr$ then $\xi_s=\ker\beta_s$ is a hyperplane field on  $N_{\epsilon''}\times S(\delta)$ that
extends the above homotopy, so we are left to see there is an almost complex structure on these hyperplane fields. (Notice that here is where we needed the term $k$ so as to guarantee that $\beta_s$ is non-singular.)
To that end notice that $\xi_s=\xi_s'+D_s$ where $D_s=\ker ((s\alpha_{std}+(1-s)\alpha_{ot}) + k(t,r)\, dr)$ 
on $S^1\times D^2$ and $\xi'_s$ is a $2(n-1)$--dimensional bundle that maps under the differential of the projection map $N_{\epsilon''}\times S(\delta)\to N_{\epsilon''}$ isomorphically onto the tangent space of $N_{\epsilon''}$. Since $d\lambda_{can}$ gives the tangent space to $N_{\epsilon}$ a $U(n-1)$ structure, we can use this isomorphism to give $\xi_s'$ an almost complex structure. Moreover, $D_s$ is an oriented 2--dimensional bundle and hence has an almost complex structure. Thus $\xi_s$ has an almost complex structure for all $s$. It is clear that where $\beta_s$ is contact this almost complex
structure agrees with the one induced by $d\beta_s$ and hence agrees with the one constructed in 
Step 2. Thus we see this is a
homotopy of almost contact structures and our proof is complete. 
\end{proof}

Using Theorem~\ref{main} we may now easily show all manifolds admitting contact structures admit ps-overtwisted ones. 
\begin{proof}[Proof of Theorem~\ref{maincor}]
Let $\xi$ be a contact structure on $M.$ In a Darboux ball inside of $M$ with coordinates $(x_1,y_1,\ldots, x_n, y_n,z)$ and contact structure $\ker (dz -\sum y_i\,dx_i)$ take a unit sphere $B$ in the $\{x_i\}$--subspace.
It is clear that $B$ is an isotropic submanifold
of $M$ with trivial conformal symplectic normal bundle. Thus we may apply
Theorem~\ref{main} to alter $\xi$ to a contact structure containing an
overtwisted family parameterized by $B.$

The statement about finding an overtwisted family with any core is proven in Theorem~\ref{anycore} below.
\end{proof}

\begin{proof}[Proof of Theorem~\ref{nonunique}]
Let $\xi=\ker(dz-\sum_{i=1}^n y_i\, dx_i)$ be the standard contact structure 
on $\R^{2n+1}$ where we are using Cartesian coordinates $(x_1,y_1,\ldots,
x_n,y_n,z).$  Let $\xi'$ be the result of performing a Lutz twist along some
embedded isotropically parameterized family of transverse curves contained in
some compact ball in $\R^{2n+1}.$  Let $B_i$ be a ball of radius $\frac 14$
about the integral points on the $z$-axis and let $\xi''$ be the result of
performing a Lutz twist along some embedded isotropically parameterized family
of transverse curves contained in each of the $B_i.$ Clearly $\xi$ can be
contact embedded in any contact $2n+1$ manifold (by Darboux's theorem), but
neither $\xi'$ nor $\xi''$ can be embedded in a Stein fillable contact structure
(like the standard contact structure on $S^{2n+1}$), thus they are not
contactomorphic to $\xi.$ Finally notice that $\xi'$ has the property that any
compact set in $\R^{2n+1}$ is contained in another compact set whose complement
can be embedded in any contact manifold, whereas $\xi'$ does not have this
property. Thus  $\xi'$ is not contactomorphic to   $\xi''.$  
\end{proof}

\begin{remark}
{\em 
From Lemma~\ref{lem:nbhd} it is almost immediate that we can construct a contact structure on
$\mathbb{R}^{2n+1}$ which has an embedded overtwisted family modeled on any core which is parameterized
by a closed embedding of $\mathbb{R}$. 
This is analogous to the unique ``overtwisted 
at infinity'' contact structure $\ker (\cos r\,  dz + r \sin r \, d\theta)$ on
$\mathbb{R}^3$ obtained by performing a ``Lutz twist along the $z$-axis'' in three
dimension.  In \cite{NiederkrugerPresas08}
a contact structure on $\mathbb{R}^{2n+1}$ is constructed which contains a
generalized overtwisted family at infinity (termed a generalized
plastikstufe in that paper). With this in mind, it would be interesting to know the
answer to the following question:}

\begin{quest}
Is there a unique contact structure on $\mathbb{R}^{2n+1}$ which contains an
embedded overtwisted family parameterized by $\mathbb{R}$ in the complement of
any compact subset of $\mathbb{R}^{2n+1}$. 
\end{quest}

\end{remark}

\section{Further Discussion}
In the first subsection below we discuss a crude form of a half Lutz twist in 
high dimensions. The problem with this form of a Lutz twist is that it involves
altering not only the contact structure on a manifold but the manifold itself. 
None the less this construction illustrates other ways the techniques in this 
paper can be used in constructing contact structures with various properties. 

From our main theorem we see that the generalized full Lutz twist discussed in
this paper does not affect the homotopy class of an almost contact structure, but
one could ask how other modifications of a contact structure near a submanifold
could affect the homotopy class of 
an almost contact structure. 
In the second subsection we discuss this issue.
In the last subsection below we consider 
which submanifolds of a contact manifold can be the core
of an overtwisted family (that is, which submanifolds can be the elliptic
singularity of an overtwisted family). 

\subsection{Generalized half Lutz twists}
It is interesting to observe that the branched cover construction in
Lemma~\ref{lem:main} cannot be used to perform a half Lutz twist. 
However one can modify that construction to perform a half Lutz twist at the
expense of changing the topology of the ambient manifold. 

Let $(M,\xi)$ be a contact manifold of dimension $2n+1$ and let $B$ be an
$(n-1)$--dimensional 
isotropic submanifold with trivial conformal symplectic normal bundle. 
According to the discussion in Section~\ref{sec:getfamily} we can, in any
neighborhood of $B,$ find a contact embedding of a neighborhood
$N_\epsilon\times S_{std}(d)$ in $M$ where $N_\epsilon$ is a neighborhood of the zero
section in $T^*B$ and $S_{std}(d)$ is the solid torus with contact structure from
Section~\ref{sec:3lutz}. Let $W$ be the unit conormal bundle for $B$ in $T^*B$
which we may think of as a submanifold of the neighborhood above. Taking an
$S^1$ from the $S_{std}(d)$ factor we see an embedding of $W\times S^1$ in the above
neighborhood. The submanifold $W\times S^1$ has a neighborhood $W\times
S^1\times D^3$ in $M.$ Let $M'$ be the manifold obtained from $M$ by removing
$W\times S^1\times D^3$ and gluing in its place $N_\epsilon \times S^1\times
S^2.$ 
\begin{proposition}
Let $(M,\xi)$ be a contact manifold of dimension $2n+1$ and let $B$ be an 
$(n-1)$--dimensional isotropic submanifold with trivial conformal symplectic
normal bundle.  With the notation above we may extend $\xi|_{M-(N_\epsilon\times
S_{std}(d))}$ over $M'$ to obtain a contact structure $\xi'$ such that for some
$\epsilon''<\epsilon$ the contact structure $\xi'$ on $N_{\epsilon''}\times S(d)$
is contactomorphic to $N_{\epsilon''}\times S_{ot}({d_\pi}).$
\end{proposition}

\begin{proof}
Following the outline in Subsection~\ref{sec:setup} we can define the desired
contact structure 
on $N_\epsilon\times S(d)$ everywhere except on $W\times
[a,b]\times[\delta'',\delta']\times S^1\times S^1.$ In the paper
\cite{GayKirby04}, Gay and Kirby construct an exact near symplectic structure on
$ [a,b]\times[\delta'',\delta']\times S^1\times S^1$ that can be used as
described at the end of Subsection~\ref{sec:setup} to try to extend the contact
structure on all of $N_\epsilon\times S(d).$ More precisely, there is a $1$--form
$\beta$ such that $d\beta$ is symplectic on $[a,b]\times[\delta'', \delta']\times
S^1\times S^1$ away from a curve $\{(c,c',\theta)\}\times S^1,$ where $c\in
(a,b), c'\in (\delta'',\delta')$ and $\theta\in S^1.$ Thus $\lambda+\beta$ is a
contact form on $W\times  [a,b]\times[\delta'',\delta']\times S^1\times S^1$
away from $W\times \{(c,c',\theta)\}\times S^1$ and has the necessary boundary
conditions to glue to the desired contact structure. Let $U$ be a neighborhood
of $W\times \{(c,c',\theta)\}\times S^1.$ It is shown in \cite{GayKirby04} that a
neighborhood of the boundary of $U$ is contactomorphic to $W\times [x,y]\times
S^1\times S^2$ with the contact form $\lambda + e^t \alpha'$ where $\alpha'$ is
a contact form on $S^1\times S^2$ giving the minimally overtwisted contact
structure (that is the one in the same homotopy class of plane fields as the
foliation of $S^1\times S^2$ by $S^2$'s). Consider  $U$ as $[-y,-x]\times
W\times S^1\times S^2$ where we use the identity diffeomorphism on most factors
and $t\mapsto -t$ on the interval factor. Notice this is an orientation
preserving diffeomorphism and the contact form in these coordinates can be
taken to be $e^t\lambda+\alpha'.$ Now we can glue in a neighborhood of the zero
section in $T^*B$ times $S^1\times S^2$ and extend the contact structure over
this (using $\lambda_{can}+\alpha'$).
\end{proof}

Since the construction given in the proposition above changes the topology of the ambient manifold, it is not reasonable to think of it as a ``real'' half Lutz twist. So we are left with the following question. 
\begin{quest}
Is there a way to perform a half Lutz twist without changing the topology of the manifold?
\end{quest}

\subsection{Almost contact structures and Lutz twists}
As discussed in Subsection~\ref{sec:almost} an obstruction to two almost contact
structures being 
homotopic is the Chern classes of the almost contact structure. In dimension 3
it is well known that Lutz twisting affects the first Chern class of the contact
structure. In higher dimensions this is not the case. 

In the proof of Theorem~\ref{main} above we showed that  the homotopy class of
almost 
contact structure is unchanged by a  Lutz twist, but one might ask if the
homotopy class can be affected with a generalized half Lutz twist (should one
ever be defined that, unlike the construction in the previous subsection, does not change the
topology of the ambient manifold). 
\begin{proposition}\label{chern}
Let $(M,\xi)$ be a closed contact $(2n+1)$-manifold for $n>1.$ Suppose $B\times
S^1$ is an embedded isotropically 
parameterized family of transverse curves in $(M,\xi)$ with $B$ of dimension
$n-1.$ If $\xi'$ is obtained from $\xi$ by altering the contact structure in a
neighborhood of  $B\times S^1$ then the Chern classes of $c_k(\xi)$ and
$c_k(\xi')$ are equal for $k<\frac{n+1}2.$
\end{proposition}
Note that the proposition implies that the first Chern class of a contact
structure 
cannot be affected by a Lutz twist except in dimension 3. 
\begin{proof}
One can easily construct a handle decomposition of $M$ in which 
a neighborhood of $B\times S^1$ can be taken to be a union of handles of index
larger than or equal to $n+1.$ Moreover the contact structures $\xi$ and $\xi'$
are the same outside a neighborhood of $B\times S^1,$ that is away from the
handles that make up the neighborhood. As $c_k$ is the primary obstruction to
the existence of a $(n-k+1)$-frame over the $2k$ skeleton of $M$ we see that
$c_k$ of $\xi$ and $\xi'$ must be the same for $2k<n+1.$
\end{proof}

In dimension 3 one can use Lutz twists to produce contact structures in any 
homotopy class of almost contact structure. One might hope to do this in higher
dimensions as well, but clearly Proposition~\ref{chern} shows our notion of Lutz
twist (even a more general one than defined here) cannot achieve this. So we ask
the following question.
\begin{quest}
Is there some other notion of Lutz twisting that affects all the Chern classes of a contact structure?
\end{quest}
Or more to the point we have the following question.
\begin{quest}
Is there some notion of Lutz twisting, or some other modification of 
a contact structure, that will guarantee that any manifold $M$ admitting a
contact structure admits one in every homotopy class of almost contact
structure?
\end{quest}

\subsection{Cores of overtwisted families}

In \cite{Presas07} the ps-overtwisted contact structures came from 
ps-overtwisted contact structures of lower dimension. More precisely, the core of
the overtwisted families constructed in dimension $2n+1$ were constructed as a
product of $S^1$ and an overtwisted family in dimension $2n-1.$ Starting in
dimension 3 where the core is just a point, one sees that all the cores of
overtwisted families constructed in \cite{Presas07} are tori of the appropriate
dimension. The ps-overtwisted contact structures constructed in
\cite{NiederkrugerVanKoert07} were constructed by taking the previous examples
and performing surgery on the ambient manifold without affecting the overtwisted
family. Thus, once again, we see that all the overtwisted families are modeled
on tori. From our construction we can show the following result. 
\begin{theorem}\label{anycore}
Let $(M,\xi)$ be a contact manifold of dimension $2n+1$. Given any
$(n-1)$--dimensional 
isotropic submanifold $B$ in $(M,\xi)$ with trivial conformal symplectic normal
bundle there is a contact structure $\xi'$ on $M$ that contains an overtwisted
family modeled on $B.$ Moreover, if $B$ is any abstract $(n-1)$--dimensional manifold (that is not necessarily already embedded in $M$)
with trivial complexified tangent bundle, then there is a contact structure on
$M$ with overtwisted family modeled on $B.$
\end{theorem}
\begin{proof}
The first statement is clear as we can find an embedded isotropically
parameterized 
family of transverse curves $B\times S^1$ as in Section~\ref{sec:getfamily} and
then use Theorem~\ref{thm:genlutz} to perform a Lutz twist to produce a contact
structure $\xi'$ with an overtwisted family modeled on $B.$

For the second statement we need to see that given a $B$ with the required 
properties 
we can embed it in $(M,\xi)$ as an isotropic submanifold with trivial conformal
symplectic normal bundle. It is clear, due to the dimensions involved, that $B$
can be embedded in a ball in $M.$ It is well known that isotropic submanifolds
of dimension less than $n$ satisfy an h-principle \cite{EliashbergMishachev02}.
This $h$-principle states that if an embedding $\psi: B\to M$ is covered by a
bundle map $TB$ to $\xi$ sending the tangent planes of $B$ to isotropic spaces
in $\xi$ then the embedding can be isotoped to an isotropic embedding. Thus we
need to construct a bundle map $TB$ to $\psi^*\xi$ sending $T_pB$ to an
isotropic subspace of $(\psi^*\xi)_p.$ In the end we will also want the
conformal symplectic normal bundle to be trivial. This implies that we need to
see a bundle isomorphism from $T(T^*B)\oplus \C$ to $\psi^*\xi.$ Since $\psi$
can be taken to have its image in a Darboux ball of $M$ we can assume that
$\psi^*\xi$ is the trivial bundle $\C^n.$ Now it is clear that if $T(T^*B)\cong
TB\otimes \C$ is trivial then we have such an isomorphism. 
\end{proof}

From this theorem we see that it is easy to produce overtwisted families modeled
on 
many manifolds. In particular, any oriented 2--manifold, respectively
3--manifold, can be realized as the core of an overtwisted family in a contact
7, respectively 9, manifold.  Moreover, the vanishing of the first Pontryagin
class of the tangent bundle of a simply connected 4--manifold is sufficient to guarantee it can
be made the core of an overtwisted family in a contact 11--manifold. (To see this, notice
that the vanishing of the Pontryagin class implies the vanishing of the second Chern class of
the complexified tangent bundle and thus there is a complex 3-frame for the 
complexified tangent bundle. Thus the complexified tangent bundle 
splits as a trivial complex 3--dimensional bundle and a line
bundle. The complex line bundle must also be trivial for if not the first Chern class of the complexified
tangent bundle would be non-zero.) It would be
very interesting to know the answer to the following question. 
\begin{quest}
If $(M^{2n+1},\xi)$ contains an overtwisted family modeled on $B$ does it also
contain 
an overtwisted family modeled on any, or even some, other $(n-1)$--manifold $B'$
(satisfying suitable tangential conditions)? 
\end{quest}
In dimension 3, overtwisted contact structures are very flexible and various 
questions about them usually have a topological flavor. That is, if something is
true topologically then it is frequently true for overtwisted contact
structures. For example, if two overtwisted contact structures are homotopic as
plane fields in dimension 3 then they are isotopic, \cite{Eliashberg89}. We also know that any
overtwisted contact structure is supported by a planar open book (just like any
3--manifold), \cite{Etnyre04b}. Thus if overtwisted families are the ``right'' generalization of
overtwisted disks to higher dimensional manifolds then we would expect to have similar
results. An affirmative answer to the question above would essentially say you
have a lot of flexibility in the cores of overtwisted families.
Theorem~\ref{anycore} is a step in that direction.


\def\cprime{$'$} \def\cprime{$'$}

\newpage

{\begin{center}{\Large \bf  Erratum to: ``On generalizing Lutz twists"}\end{center}}

\bigskip

\bigskip

The proof of Lemma~3.4 in \cite{EtnyrePancholi11} is incorrect. Below we will describe the problem with the proof and then show how it can easily be repaired in dimension 5. We then observe that Lemma~3.4, and thus the main results of the paper, is true in all dimensions based on recent work of Borman, Eliashberg and Murphy \cite{BormanEliashbergMurphy}. However this approach does not give an explicit construction and hence goes against the sprit of the original paper and in addition all the results  of \cite{EtnyrePancholi11} follow directly from \cite{BormanEliashbergMurphy}.


\noindent
{\em Acknowledgement:} We thank Yasha Eliashberg for pointing out the error in the proof of Lemma~3.4 in \cite{EtnyrePancholi11}. 
The first author was partially supported by a grant from the Simons Foundation (\#342144) and NSF grant DMS-1309073.

\section{Exact Lagrangians, Liouville flows, and the error in the proof of Lemma~3.4}
We begin by recalling the statement of Lemma~3.4 from \cite{EtnyrePancholi11}. To state the lemma we first establish some notation (that is slightly different that what was used in \cite{EtnyrePancholi11}). Consider $T^2\times [0,1]$ with coordinates $(\theta,
\phi, r)$ and the contact structure $\xi_i = \ker \alpha_i$, $i=1,2$, where
\[
\alpha_i= k_i(r)\, d\theta + l_i(r)\, d\phi.
\]
Here we have $k_1(r)=\cos \frac\pi 2 r$ and $l(r)=\sin \frac\pi 2 r$, and for $i=2$ we have $k_2$ and $l_2$ agreeing with $k_1$ and $l_1$ near $r=0$ and $1$, and the curve $(k_2(r),l_2(r))$ in $\R^2$ has $5\pi/2$ winding about the origin. In particular notice that $\xi_2$ is obtained from $\xi_1$ by adding Giroux torsion. Lemma~3.4 from \cite{EtnyrePancholi11} now reads as follows. 

\begin{lemma}\label{RMN1} 
Let $W$ be a manifold with contact form $\lambda$, there is a contact structure on $W \times [0,1] \times ([0,1]  \times T^2)$ such that the following properties are satisfied:
\begin{enumerate}
\item near $W\times \{0\}\times [0,1]\times T^2$  and $W\times [0,1]\times  \{0,1\}\times T^2$ the contact structure is  contactomorphic to $\lambda+e^t \alpha_1,$ and
\item near $W\times \{1\}\times [0,1]\times T^2$ the contact structure is contactomorphic to $\lambda+e^t \alpha_2.$ 
\end{enumerate}
Here $t$ is the coordinate on the first $[0,1]$ factor.
\end{lemma}

See \cite{EtnyrePancholi11} for details on how the main constructions and theorems of the paper follow from this lemma.
The strategy of the proof in \cite{EtnyrePancholi11} was:
\begin{enumerate}
\item To construct a contact structure on $W\times [0,1]\times T^3$ that near $W\times\{0\}\times T^3$ is given by $\lambda+e^t\beta_0$ and near  $W\times\{1\}\times T^3$ is given by $\lambda +e^t\times \beta_1$, where $\beta_i$ is the contact structure on $T^3$ with Giroux torsion $i$ and we are thinking of $T^3$ as $S^1\times T^2$ with the $S^1$-factors Legendrian curves.
\item Then cut $W\times [0,1]\times T^3$ along $W\times [0,1]\times (\{\theta_0,\theta_1\}\times T^2)$ so that one of the resulting pieces is as described in the lemma. 
\end{enumerate}
To try to arrange this let $\beta=p_1\, d\theta_1+p_1\, d\theta_2$ be the Liouville form on  $T^*T^2=\R^2\times T^2$ with coordinates $(p_1,p_2,\theta_1,\theta_2)$. Notice that $\alpha=\lambda+\beta$ is a contact form on $W\times T^*T^2$. We will see below that we can arrange the two items above that are needed for our proof if there is a radial vector field $v$ in $\R^2$ centered at a point $p$ whose flow expands $d\beta$ (that is, $L_vd\beta=d\beta$) and a Lagrangian torus $T^2$ in a small neighborhood of $\{q\}\times T^2\subset T^*T^2$ that is exact with respect to $\iota_vd\beta$ that is isotopic to $\{q\}\times T^2$ by an isotopy disjoint from $\{p\}\times T^2$. One may easily arrange all of this except for either the last requirement of disjointness or the exactness of the Lagrangian torus. In \cite{EtnyrePancholi11} we assumed this could be arranged (though in the presentation there it was not clear these were precisely the conditions necessary), but to the best of our knowledge this cannot be done. More explicitly in \cite{EtnyrePancholi11} we took the Lagrangian torus $T^2\times\{(0,0)\}$ and the radial vector field $v$ to be centered at a point disjoint from the origin. Notice that the torus is exact with respect to $\beta$ but not with respect to $\iota_vd\beta$, and thus the construction does not work.  In the next section we will see that the condition of having an exact Lagrangian torus can be removed in the 5 dimensional setting. (For simplicity or presentation below we will take $p=(0,0)$.)

Let $X=D\times T^2$ where $D$ is a disk of radius $R$ about the origin in $\R^2$ and $R>2$ is some constant. The form $\beta$ restricts to the contact form $\beta_0$ on $T^3=\partial X$. Now let $X'=D'\times T^2$ where $D'$ is a small disk about the origin in $\R^2$.   By noticing that since the radial vector field $v=p_1\frac{\partial}{\partial p_1}+ p_2\frac{\partial}{\partial p_2}$ is the Liouville field for $\beta$ on $T^*T^2$ one easily sees that $\overline{X-X'}$ is a piece of the symplectization of the minimally twisting tight contact structure on $T^3$. 

Let $\gamma_1$ be the intersection of $D$ with the ray  leaving the origin in $\R^2$ that forms an angle of $\pi/2$ with the positive $p_1$-axis and similarly let $\gamma_2$ be the intersection of $D$ with the ray in $\R^2$ that forms an  angle of $-3\pi/4$ with the positive $p_1$-axis. Setting $Y$ equal to the component of $\overline{X-X'}$ cut along the $\gamma_i\times T^2$ that contains points lying above the negative $p_1$-axis in $\R^2$, one easily sees that $Y$ is a piece of the symplectization of $\alpha_{1}$ (using the notation from above) on $T^2\times [0,\pi/2]$ (here we have rescaled $[0,1]$ to $[0,\pi/2]$) and the parts of $Y$ lying above $\gamma_1$ and $\gamma_2$ are the ``vertical"  or flat boundaries of the symplectization). 

If there is an exact Lagrangian torus $T^2$ in $T^*T^2$ as discussed above then Lemma~3.6 in \cite{EtnyrePancholi11} shows there is an embedding  $\Phi\co W\times T^2\to W\times T^*T^2$ so that $\Phi(W\times T^2)$ is contact, disjoint from $Z=W\times \{(0,0)\}\times T^2$ and isotopic to $F_q=W\times \{q\}\times T^2$ in the complement of (a neighborhood of) $Z$. (For convenience we take  $q$ to be  a point on the negative $p_1$-axis in $D$. If this were not the case we might need to re-choose the $\gamma_i$.) We can then let $C$ be the 2-fold cover of $W\times X$ branched over $\Phi(W\times T^2)$. It is well known that $C$ has a contact structure that away from the branched locus is just the lift of the contact structure on $W\times X$, see for example \cite[Theorem~7.5.4]{Geiges08}. Moreover it is clear that the cover is diffeomorphic to $W \times X$. Notice that the boundary of $C$ is $W\times T^3$ and in a neighborhood of the boundary the contact structure is simply $W$ times a piece of the symplectization of the Giroux torsion 1 contact structure on $T^3$, so that the boundary of $C$ is $W$ times the convex end of the piece of the symplectization. 

Notice that we can take $W\times X'$ to be a neighborhood of $Z$ in $W\times X$ that is disjoint from $\Phi( W\times T^2)$ and the isotopy of $\Phi( W\times T^2)$ to $F_q$. Clearly $W\times X'$ lifts to two disjoint copies of $W\times X'$ in $C$. Let $N$ be one of these and set $C'=C\setminus N$. It is clear that $\partial C'-\partial C$ has a neighborhood in $C'$ where the contact structure looks like $W$ times a piece of the symplectization of the standard minimally twisting contact structure on $T^3$, so that the boundary component is $W$ times the concave end of the piece of the symplectization. Furthermore notice that each $\gamma_i\times T^2$ lifts to two copies in $C'$. The copies that intersect with $N$ will divide $C'$ into two pieces. Let $C''$ be the piece that contains the branch locus. Notice that the branched covering map restricted to a neighborhood of $(\partial C''-((\partial C)\cap C''))$ in $C''$ is mapped diffeomorphically (and contactomorphically)  to a neighborhood of $(\partial Y-((\partial X)\cap Y))$ in $Y$. Moreover the remaining boundary component of $C''$ can easily be seen to have a neighborhood that is contactomorphic to $W$ times a piece of the symplectization of $\alpha_2$. From this discussion it should be clear that the contact structure on $W''$ is the structure described in the lemma.

\section{Fixing the error in dimension 5}
In this section we show how to fix the proof of  Lemma~3.4 from \cite{EtnyrePancholi11} in the 5 dimensional case. 

\begin{lemma}\label{RMN11} 
There is a contact structure on $S^1 \times [0,1] \times ([0,1]  \times T^2)$ such that the following properties are satisfied:
\begin{enumerate}
\item\label{1} near $S^1\times \{0\}\times [0,1]\times T^2$  and $S^1\times [0,1]\times  \{0,1\}\times T^2$ the contact structure is  contactomorphic to $d\theta+e^t \alpha_1,$ and
\item\label{2} near $S^1\times \{1\}\times [0,1]\times T^2$ the contact structure is contactomorphic to $d\theta+e^t \alpha_2,$ 
\end{enumerate}
where $\theta$ is the angular coordinate on $S^1$ and $t$ is the coordinate on the first $[0,1]$ factor.
\end{lemma}

\begin{proof}
We will be considering $S^1\times T^*T^2$ with the contact structure $\alpha=d\theta+p_1\, d\theta_1+ p_2\, d\theta_2$. 
From the discussion in the previous section we only need to check that there is an embedding $\Phi\co T^3\to S^1\times T^*T^2$ so that $\Phi(T^3)$ is contact, disjoint from $Z=S^1\times \{(0,0)\}\times T^2$ and isotopic to $F_q=S^1\times \{q\}\times T^2$ in the complement of (a neighborhood of) $Z$ where $q=(-1-\epsilon,0)$, for some small $\epsilon>0$, is a point in $\R^2$.  

Using coordinates $(\phi,\phi_1,\phi_2)$ on $T^3$ we define
\[
\Phi (\phi,\phi_1,\phi_2)= (\phi+\phi_2, \sin \phi, -1-\epsilon+ \cos\phi, \phi_1,\phi_2).
\]
Now we see
\begin{align*}
\beta= \Phi^*\alpha& = d\phi+ d\phi_2+ (\sin\phi)\, d\phi_1+ (-1-\epsilon)\, d\phi_2+(\cos\phi)\, d\phi_2\\
&= d\phi + (\sin\phi) \, d\phi_1 + (\cos\phi)\, d\phi_2 - \epsilon\, d\phi_2.
\end{align*}
and
\[
d\beta= (\cos\phi)\, d\phi\wedge d\phi_1 - (\sin\phi)\, d\phi\wedge d\phi_2.
\]
Thus
\begin{align*}
\beta\wedge d\beta&=(\sin^2\phi+\cos^2\phi)\, d\phi\wedge d\phi_1\wedge d\phi_2 -\epsilon(\cos\phi)\,  d\phi\wedge d\phi_1\wedge d\phi_2 \\
&= (1-\epsilon\cos\phi)\,  d\phi\wedge d\phi_1\wedge d\phi_2.
\end{align*}
Since $(1-\epsilon\cos\phi)>0$ we have a contact embedding. Also note 
\[
\Phi_\delta (\phi,\phi_1,\phi_2)= (\phi+\phi_2, \delta(\sin \phi), -1-\epsilon+ \delta(\cos\phi), \phi_1,\phi_2)
\]
is an isotopy from $\Phi$ to a map with image the $T^3$ above $(0,-1-\epsilon)$ in $\R^2$ and the isotopy is disjoint from the $T^3$ above $(0,0)$.
\end{proof}

\section{Overtwisted contact structure approach}\label{ota}
In this section we show that Lemma~3.4 from \cite{EtnyrePancholi11}, recalled as Lemma~\ref{RMN1} above, is indeed true due to Borman, Eliashberg, and Murphy's recent breakthrough \cite{BormanEliashbergMurphy}. 

We first note that  Lemma~\ref{RMN1} explicitly defines a contact structure near the boundary of $W \times [0,1] \times ([0,1]  \times T^2)$. It is easy to check that $\alpha_2$ and $\alpha_1$ are homotopic, rel boundary, as plane fields, {\em cf.\ }\cite[Lemma~4.5.3]{Geiges08}. Let $\alpha_t$ be the homotopy. Now $e^{-t}\lambda +\alpha_{f(t)}$, for some function $f(t)$, extends the contact form from a neighborhood of the boundary of $W \times [0,1] \times ([0,1]  \times T^2)$ to a nonsingular form on the whole manifold.  Moreover its kernel splits as $\xi^2\oplus \xi'$ where $\xi^2$ is contained in the tangent space of $[0,1]\times T^2$ and $\xi'$ projects isomorphically onto the tangent space of $W\times [0,1]$. Thus $\xi'$ inherits a complex structure from  $e^{-t} \lambda$ and $\xi^2$ inherits one as an oriented plane field. Thus we have constructed an almost contact structure on $W \times [0,1] \times ([0,1]  \times T^2)$ that extends our given contact structure. The main result of \cite{BormanEliashbergMurphy} implies this almost contact structure is homotopic to an 
actual contact structure by a homotopy that is fixed outside any open neighborhood of the ``non-contact" region. The resulting contact structure can be taken to be the one promised by Lemma~3.4 in \cite{EtnyrePancholi11}. 

\def\cprime{$'$} \def\cprime{$'$}

\end{document}